\documentclass[12pt]{article}
\usepackage{amssymb,amsfonts,amsmath,amsthm}

\newtheorem{theorem}{Theorem}
\newtheorem{proposition}[theorem]{Proposition}
\newtheorem{lemma}[theorem]{Lemma}
\newtheorem{corollary}[theorem]{Corollary}

\newcommand{\calc}{{\cal C}}
\newcommand{\N}{\mathbb{N}}
\newcommand{\Z}{\mathbb{Z}}
\newcommand{\dd}{\displaystyle }
\newcommand{\nn}{\N}

\newcommand{\bld}[2]{{\buildrel{#1}\over{#2}}}
\newcommand{\st}[2]{{\mathrel{\mathop{#2}\limits_{#1}}{}\!}}
\newcommand{\stb}[3]{{\st{{#1}}{\bld{{#2}}{#3}}{}\!}}
\newcommand{\xmare}[2]{\stb{#1}{#2}{\mbox{\Huge$\times$}}}

\newcommand{\s}{{\sigma}}

\title{\bf Finite groups determined\\ by an inequality of the orders\\ of their normal subgroups}
\author{Marius T\u arn\u auceanu}
\date{October 1, 20011}

\begin{document}

\maketitle

\begin{abstract}
    In this article we introduce and study a class of finite groups
    for which the orders of normal subgroups satisfy a certain inequality.
    It is closely connected to some well-known arithmetic classes of natural numbers.
\end{abstract}



\section{Introduction}

Let $n$  be a natural number and $\s(n)$ be the sum of all
divisors of $n$. We say that $n$ is a {\it deficient number} if
$\s(n)<2n$ and a {\it perfect number} if $\s(n)=2n$ (for more
details on these numbers, see \cite{5}). Thus, the set consisting
of both the deficient numbers and the perfect numbers can be
characterized by the inequality
\[ \dd\sum_{d\in L_n}d\le 2n \,, \]
where $L_n=\{d\in\nn\mid d|n\}.$

Now, let $G$ be a finite group. Then the set $L(G)$ of all
subgroups of $G$ forms a complete lattice with respect to set
inclusion, called the {\it subgroup lattice} of $G$. A remarkable
subposet of $L(G)$ is constituted by all cyclic subgroups of $G$.
It is called the {\it poset of cyclic subgroups} of $G$ and will
be denoted by $C(G)$. If the group $G$ is cyclic of order $n$,
then $L(G)=C(G)$ and they are isomorphic to the lattice $L_n$. So,
$n$ is deficient or perfect if and only if
\begin{equation}
    \sum_{H\in L(G)}|H|\le 2|G| \,, \tag{1}
\end{equation}
or equivalently
\begin{equation}
    \sum_{H\in C(G)}|H|\le 2|G| \,. \tag{2}
\end{equation}
\bigskip

In \cite{1} we have studied the classes $\calc_{1}$ and
$\calc_{2}$ consisting of all finite groups $G$ which satisfy the
inequalities $(1)$ and $(2)$, respectively. The starting point for
our discussion is given by the open problem in the end of
\cite{1}. It suggests us to extend the initial condition $(1)$ in
another interesting way, namely
\begin{equation}
    \sum_{H\in N(G)}|H|\le 2|G| \,, \tag{3}
\end{equation}
where $N(G)$ denotes the set of normal subgroups of $G$. Recall
that $N(G)$ forms a sublattice of $L(G)$, called the {\it normal
subgroup lattice} of $G$. We also have $L(G)=N(G)$, for any finite
cyclic group $G$. Hence, in the same manner as above, one can
introduce the class $\calc_{3}$ consisting of all finite groups
$G$ which satisfy the inequality $(3)$. Clearly, it properly
contains $\calc_{1}$ (the symmetric group $S_3$ belongs to
$\calc_{3}$ but not to $\calc_{1}$) and is different from
$\calc_{2}$ (the dihedral group $D_8$ belongs to $\calc_{2}$ but
not to $\calc_{3}$). Characterizing finite groups in $\calc_{3}$
is difficult, since the structure of the normal subgroup lattice
is unknown excepting few particular cases. Their investigation is
the main goal of this paper.

The paper is organized as follows. In Section 2 we study some
basic properties of the class $\calc_{3}$, while Section 3 deals
with several classes of finite groups that belong to $\calc_{3}$.
The most significant results are obtained for nilpotent groups,
nonabelian $P$-groups, metacyclic groups and solvable T-groups.

Most of our notation is standard and will not be repeated here.
Basic definitions and results on groups can be found in \cite{4}
and \cite{7}. For subgroup lattice concepts we refer the reader to
\cite{6} and \cite{8}.

\section{Basic properties of the class $\calc_{3}$}

For a finite group $G$ let us denote
\[ \s_{3}(G)=\sum_{H\in N(G)}\,\frac{|H|}{|G|}=\sum_{H\in N(G)}\,\frac{1}{|G:H|} \,. \]
In this way, $\calc_{3}$ is the class of all finite groups $G$ for
which $\s_{3}(G)\le 2$. First of all, observe that for Dedekind
groups (that is, groups with all subgroups normal) this function
coincides with the function $\s_{1}$ defined and studied in
\cite{1}. In this way, a finite Dedekind group belongs to
$\calc_{3}$ if and only if it belongs to $\calc_{1}$. $\s_{3}$ is
also a multiplicative function: if $G$ and $G'$ are two finite
groups satisfying $\gcd(|G|, |G'|)=1$, then
\[ \s_{3}(G \times G')=\s_{3}(G)\s_{3}(G') \,. \]
By a standard induction argument, it follows that if $G_{i}$,
$i=1,2,\dots,k$, are finite groups of coprime orders, then
\[ \s_{3}(\xmare{i=1}k G_{i})=\prod_{i=1}^{k} \s_{3}(G_{i}) \,. \]

Obviously, $\calc_{3}$ contains the finite cyclic groups of prime
order. On the other hand, we easily obtain
$$\s_{3}(\Z_{p}\times\Z_{p}) = \frac{1+p+2p^2}{p^2}>2, \hspace{1mm} {\rm
for}\hspace{1mm} {\rm any}\hspace{1mm} {\rm prime}\hspace{1mm}p \,
.$$ This relation shows that $\calc_{3}$ is not closed under
direct products or extensions.

In order to decide whether the class $\calc_{3}$ is closed under
subobjects, i.e.\@ whether all subgroups of a group in $\calc_{3}$
also belong to $\calc_{3}$, the notion of $P$-group (see  \cite{6}
and \cite{8}) is very useful. Let $p$ be a prime, $n\geq2$ be a
cardinal number and $G$ be a group. We say that $G$ belongs to the
class $P(n,p)$ if it is either elementary abelian of order $p^n$,
or a semidirect product of an elementary abelian normal subgroup
$H$ of order $p^{n-1}$ by a group of prime order $q\neq p$ which
induces a nontrivial power automorphism on $H$. The group $G$ is
called a $P$-$group$ if $G\in P(n,p)$ for some prime $p$ and some
cardinal number $n\geq 2$. It is well-known that the class
$P(n,2)$ consists only of the elementary abelian group of order
$2^n$. Also, for $p>2$ the class $P(n,p)$ contains the elementary
abelian group of order $p^n$ and, for every prime divisor $q$ of
$p-1$, exactly one nonabelian $P$-group with elements of order
$q$. Moreover, the order of this group is $p^{n-1}q$ if $n$ is
finite. The most important property of the groups in a class
$P(n,p)$ is that they are all lattice-isomorphic (see Theorem
2.2.3 of \cite{6}).

Now, let $p,q$ be two primes such that $p\neq 2$, $q\mid p-1$ and
$p^2q \geq 1+p+2p^2$ (for example, $p=7$ and $q=3$), and let $G$
be the nonabelian $P$-group of order $p^2q$. Then $N(G)$ consists
of $G$ itself and of all subgroups of the elementary abelian
normal subgroup $H\cong\Z_{p}\times\Z_{p}$. This implies that
$$\s_{3}(G)=\frac{1+p+2p^2+p^2q}{p^2q}\leq 2,$$ that is $G$ belongs to
$\calc_{3}$. Since $H$ is not contained in $\calc_{3}$, we infer
that $\calc_{3}$ is not closed under subobjects.  On the other
hand, we know that $L(G)$ is isomorphic to the subgroup lattice of
the elementary abelian group of order $p^3$, which not belongs to
$\calc_{3}$. So, $\calc_{3}$ is not closed under lattice
isomorphisms, too.

Finally, let $G$ be a group in $\calc_{3}$ and $N$ be a normal
subgroup of $G$. Then we easily get
$$\s_{3}(G/N)=\dd\sum_{^{H\in N(G)}_{N\subseteq
H}}\,\frac1{|G:H|}\leq\s_{3}(G)\leq2,$$ proving that $\calc_{3}$
is closed under homomorphic images.

\section{Finite groups contained in $\calc_{3}$}

In this section we shall focus on characterizing some particular
classes of groups in $\calc_{3}$. The simplest case is constituted
by finite $p$-groups.

\begin{lemma}\label{th:C1}
    A finite $p$-group is contained in $\calc_{3}$ if and only if it is cyclic.
\end{lemma}

\begin{proof}
Let $G$ be a finite $p$-group of order $p^n$ which is contained in
$\calc_{3}$ and suppose that it is not cyclic. Then $n\geq 2$ and
$G$ possesses at least $p+1$ normal subgroups of order $p^{n-1}$.
It results
$$\s_{3}(G)\geq \frac{1+(p+1)p^{n-1}+p^n}{p^n}>2,$$ therefore $G$
does not belong to $\calc_{3}$, a contradiction.

Conversely, for a finite cyclic $p$-group $G$ of order $p^n$, we
obviously have
$$\s_{3}(G)=\frac{1+p+...+p^n}{p^n}=\frac{p^{n+1}-1}{p^{n+1}-p^n}\leq2.$$
\end{proof}

The above lemma leads to a precise characterization of finite
nilpotent groups contained in $\calc_{3}$. It shows that the
finite cyclic groups of deficient or perfect order are in fact the
unique such groups.

\begin{theorem}\label{th:C1}
    Let $G$ be a finite nilpotent group.
    Then $G$ is contained in $\calc_{3}$ if and only if it is cyclic
    and its order is a deficient or perfect number.
\end{theorem}

\begin{proof}
Assume that $G$ belongs to $\calc_{3}$ and let $\xmare{i=1}kG_{i}$
be its decomposition as a direct product of Sylow subgroups. Since
$G_i$, $i=1,2,...,k$, are of coprime orders, one obtains
$$\s_{3}(G)=\prod_{i=1}^{k} \s_{3}(G_{i})\leq 2.$$ This inequality
implies that $\s_{3}(G_{i})\leq 2$, for all $i=\overline{1,k}$.
So, each $G_{i}$ is contained in $\calc_{3}$ and it must be
cyclic, by Lemma 1. Therefore $G$ itself is cyclic and $\mid G
\mid$ is a deficient or perfect number.

The converse is obvious, because a finite cyclic group of
deficient or perfect order is contained in $\calc_{1}$ and hence
in $\calc_{3}$.
\end{proof}

Since $\calc_{3}$ is closed under homomorphic images and the
quotient $G/G'$ is abelian for any group $G$, the next corollary
follows immediately from the above theorem.

\begin{corollary}
    Let $G$ be a finite group contained in $\calc_{3}$.
    Then $G/G'$ is cyclic and its order is a deficient or perfect number.
\end{corollary}

Theorem 2 also shows that in order to produce examples of
noncyclic groups contained in $\calc_{3}$, we must look at some
classes of finite groups which are larger than the class of finite
nilpotent groups. One of them is constituted by the finite
supersolvable groups and an example of such a group that belongs
to $\calc_{3}$ has been already given: the nonabelian $P$-group of
order $p^2q$, where $p,q$ are two primes satisfying $p\neq 2$,
$q\mid p-1$ and $p^2q \geq 1+p+2p^2$. In fact, $\calc_{3}$
includes only a small class of groups of this type, as shows the
following proposition.

\begin{proposition}
    Let $p,q$ be two primes such that $p\neq 2$ and $q\mid p-1$.
    Then the finite nonabelian $P$-group $G$ of order $p^{n-1}q$ is contained in $\calc_{3}$ if and only if either $n=2$ or $n=3$ and $p^2q \geq 1+p+2p^2$.
\end{proposition}

\begin{proof}
By Lemma 2.2.2 of \cite{6}, the derived subgroup $G'$ of $G$ is
elementary abelian of order $p^{n-1}$ and $N(G)$ consists of $G$
itself and of the subgroups of $G'$. For every $k=0,1,...,n-1$,
let us denote by $a_{n-1,p}(k)$ the number of all subgroups of
order $p^k$ of $G'$. Then we have
$$\s_{3}(G)=\frac{x_{n-1,p}+p^{n-1}q}{p^{n-1}q}\hspace{1mm},$$ where $x_{n-1,p}=\dd\sum_{k=0}^{n-1}p^ka_{n-1,p}(k)$.
Mention that the numbers $a_{n-1,p}(k)$ satisfy the following
recurrence relation
$$a_{n-1,p}(k)=a_{n-2,p}(k)+p^{n-1-k}a_{n-2,p}(k-1), \hspace{1mm}{\rm
for}\hspace{1mm} {\rm all}\hspace{1mm} k=\overline{1,n-2},$$ and
have been explicitly determined in \cite{9}. Denote by $a_{n-1,p}$
the total number of subgroups of $G'$, that is
$a_{n-1,p}=\dd\sum_{k=0}^{n-1}a_{n-1,p}(k)$. One obtains that
$x_{n-1,p}$ satisfies also a certain recurrence relation, namely
$$x_{n-1,p}=x_{n-2,p}+p^{n-1}a_{n-2,p}, \hspace{1mm}{\rm
for}\hspace{1mm} {\rm all}\hspace{1mm} n\geq 2\hspace{1mm}.$$ Then
$$x_{n-1,p}=1+\dd\sum_{k=1}^{n-1}p^ka_{k-1,p}\hspace{1mm}.$$ For
$n\geq 4$ we get $a_{n-2,p}\geq a_{2,p}=p+3$, therefore
$$x_{n-1,p}>p^{n-1}(p+3)>p^n>p^{n-1}q.$$ In other words, we have
$\s_{3}(G)>2$, i.e. $G$ is not contained in $\calc_{3}$. For $n=3$
it results $x_{2,p}=1+p+2p^2$, which implies that $G$ belongs to
$\calc_{3}$ if and only if $p^2q \geq 1+p+2p^2$. Obviously, for
$n=2$ we have $x_{1,p}=1+p\leq pq$ and hence $G$ is contained in
$\calc_{3}$.
\end{proof}

Remark that the quotient $G/G'$ is cyclic of deficient or perfect
order for all finite nonabelian $P$-groups $G$, but they are not
always contained in the class $\calc_{3}$, as shows Proposition 4.
In this way, the necessary condition on $G$ in Corollary 3 is not
sufficient to assure its containment to $\calc_{3}$.
\bigskip

A remarkable class of finite supersolvable groups is constituted
by the metacyclic groups. From Lemma 2.1 in \cite{3}, such a group
$G$ has a presentation of the form
\begin{equation}
    <x,y |\hspace{1mm} x^k=y^l, y^m=1, y^x=y^n>, \tag{4}
\end{equation}
where $k$, $l$, $m$ and  $n$ are positive integers such that $m |
\hspace{1mm}(n^k-1)$ and $m | \hspace{1mm}l(n-1)$. Moreover,
$N=\hspace{1mm}<y>$ is a normal subgroup of $G$,
$G/N=\hspace{1mm}<xN>$ is of order $k$ and
$G'=\hspace{1mm}<y^{n-1}>$. In several cases we are able to decide
when $G$ is contained in $\calc_{3}$.

Suppose first that $G$ belongs to $\calc_{3}$. Then $G/G'$ is
cyclic and its order is a deficient or perfect number, by
Corollary 3. We infer that $G'=N$. This implies that
gcd$(m,n-1)=1$ and so $m |\hspace{1mm}l$. It follows that $G$ has
a presentation of the form
\begin{equation}
    <x,y |\hspace{1mm} x^k=y^m=1, y^x=y^n>, \tag{5}
\end{equation}
where gcd$(m,n-1)=1$, $m |\hspace{1mm} \dd\frac{n^k-1}{n-1}$ and
$k$ is a deficient or perfect number. We also remark that $G$ is a
split metacyclic group with trivial center. Giving a precise
description of the normal subgroup lattice of such a group is very
difficult, but clearly $G$ itself and all subgroups of $G'$ are
contained in $N(G)$. In this way
\begin{equation}
    \s_{3}(G)=\sum_{H\in N(G)}|H|\geq km+\s(m),\nonumber
\end{equation}
which implies that
\begin{equation}
    \s(m)\leq km. \nonumber
\end{equation}
So, we have proved the following proposition.

\begin{proposition}
    Let $G$ be the finite metacyclic group given by {\rm(4)}. If $G$
    belongs to $\calc_{3}$, then it is a split metacyclic group of the form
    {\rm(5)} and $\s(m)\leq km$.
\end{proposition}

The necessary conditions established in the above proposition
become sufficient under the supplementary assumption that $G/G'$
is of prime order, that is $k$ is a prime. In this case the normal
subgroup lattice of the group $G$ of type (5) is given by the
equality
\begin{equation}
    N(G)=\{G\}\cup L(G'). \nonumber
\end{equation}
Indeed, if $H$ is a normal subgroup of $G$ which is not contained
in $G'$, then $G'\subset HG'$ and therefore $HG'=G$. Since $G'$ is
cyclic, one obtains that $G/H\cong G'/H\cap G'$ is also cyclic.
Thus $G'\subseteq H$ and hence $H=G$.

We have $$\s_{3}(G)=km+\s(m)\leq 2|G|\Leftrightarrow \s(m)\leq
km$$ and so the following theorem holds.

\begin{theorem}
    Let $G$ be the finite metacyclic group given by {\rm(4)} and assume that $k$ is a prime. Then $G$
    belongs to $\calc_{3}$ if and only if it is a split metacyclic group of the form
    {\rm(5)} and $\s(m)\leq km$.
\end{theorem}

Remark that in the particular case when $k=2$ and $n=m-1$ the
group with the presentation (5) is in fact the dihedral group
$D_{2m}$. In this way, by Theorem 6 we obtain that $D_{2m}$ is
contained in $\calc_{3}$ if and only if gcd$(m,m-2)=1$ (that is,
$m$ is odd) and $\s(m)\leq 2m$ (that is, $m$ is deficient or
perfect).

\begin{corollary}
    The finite dihedral group $D_{2m}$ is contained in $\calc_{3}$ if and only if $m$ is an odd deficient or perfect number.
\end{corollary}

The dihedral groups $D_{2m}$ with $n$ odd satisfy the property
that they have no two normal subgroups of the same order, that is
the order map from $N(D_{2m})$ to $L_{2m}$ is one-to-one. In fact,
we can easily see that if a finite group satisfy this property and
its order is a deficient or perfect number, then it belongs to
$\calc_{3}$. By Corollary 7 it also follows that the dihedral
groups of type $D_{2p^n}$, with $p$ an odd prime, are all
contained in $\calc_{3}$. These groups satisfy the stronger
property that their normal subgroup lattices are chains, a
condition which is sufficient to assure the containment of an
arbitrary finite group to $\calc_{3}$. In particular, we remark
that $\calc_{3}$ also contains some finite groups with few normal
subgroups, as finite simple groups or finite symmetric groups.
\bigskip

Another class of finite groups which can naturally be connected to
$\calc_{3}$ is constituted by finite T-groups. Recall that a group
$G$ is called a {\it T-group} if the normality is a transitive
relation on $G$, that is if $H$ is a normal subgroup of $G$ and
$K$ is a normal subgroup of $H$, then $K$ is normal in $G$ (in
other words, every subnormal subgroup of $G$ is normal in $G$).
The structure of arbitrary finite T-groups is unknown, without
some supplementary assumptions. One of them is constituted by the
solvability. The finite solvable T-groups have been described in
Gasch\"{u}tz \cite{2}: such a group $G$ possesses an abelian
normal Hall subgroup $N$ of odd order such that $G/N$ is a
Dedekind group (note that $G/N$ is the unique maximal nilpotent
quotient of $G$). Moreover, it is well-known that a finite
solvable T-group is metabelian and all its subgroups are also
T-groups.

Suppose first that the finite solvable T-group $G$ belongs to
$\calc_{3}$ and let $H$ be a complement of $N$ in $G$ (that is,
$NH=G$ and $N\cap H=1$). Set $|G|=2^kr$, where $r$ is an odd
number. Then $|N|=r$ and $|H|=2^k$. Since $\calc_{3}$ is closed
under homomorphic images, it follows that $H\cong G/N$ is also
contained in $\calc_{3}$. We infer that $H$ is cyclic, in view of
Theorem 2.

By Remark 4.1.5 (page 160) of \cite{6}, we are able to describe
the subgroup lattice of $G$
\begin{equation}
L(G)=\{N_1H_1^x \mid x \in N, N_1\leq N, H_1\leq H \}, \nonumber
\end{equation}
where $N_1H_1^x\leq N_2H_2^y$ if and only if $N_1\leq N_2$,
$H_1\leq H_2$ and $xy^{-1} \in C_N(H_1)N_2$. The normal subgroup
lattice of $G$ can be easily determined by using the above
equality
$$N(G)=\{N_1H_1^x \in L(G) \mid C_N(H_1)N_1=N\}.$$ Remark that all
subgroups of $G$ which are contained in $N$ or contain $N$ are
normal in $G$. Then
\begin{equation}
\s_{3}(G)=\sum_{K\in N(G)}|K| \geq
\s_3(N)+|N|\hspace{1mm}(\s_3(G/N)-1)=\s_3(N)+r(2^{k+1}-2).\nonumber
\end{equation}
On the other hand, we have
\begin{equation}
\s_3(G)\leq 2|G|=2^{k+1}r \nonumber
\end{equation}
because $G$ belongs to $\calc_{3}$. By the previous two
inequalities, one obtains that $\s_3(N)\leq 2|N|$, that is $N$ is
also contained in $\calc_{3}$. Therefore $N$ is cyclic and its
order $r$ is a deficient or perfect number. Hence we have proved
the following result.

\begin{proposition}
    Let $G$ be a finite solvable T-group. If $G$
    belongs to $\calc_{3}$, then it possesses a cyclic normal Hall
    subgroup $N$ of odd deficient or perfect order and every
    complement of $N$ in $G$ is also cyclic. In particular, $G$
    is a split metacyclic group.
\end{proposition}

\begin{corollary}
    All finite solvable T-groups contained in $\calc_{3}$ are extensions
    of a cyclic group of odd order by a cyclic {\rm 2}-group, both contained in $\calc_{3}$.
\end{corollary}
\bigskip

As show all our previous results, for several particular classes
of finite groups $G$ we are able to give necessary and sufficient
conditions such that $G$ belongs to $\calc_{3}$. Finally, we note
that the problem of finding characterizations of {\it arbitrary}
finite groups contained in $\calc_{3}$ remains still open.

\vspace*{5ex}

\small

\hfill
\begin{minipage}[t]{5cm}
Marius T\u arn\u auceanu \\
Faculty of  Mathematics \\
``Al.I. Cuza'' University \\
Ia\c si, Romania \\
e-mail: {\tt tarnauc@uaic.ro}
\end{minipage}

\end{document}